\documentclass[11pt,a4paper,leqno]{amsart}
\usepackage[T1]{fontenc}
\usepackage[utf8]{inputenc}
\usepackage{lmodern}
\usepackage[a4paper,margin=28mm]{geometry}
\usepackage{amsmath,amssymb,amsthm,mathrsfs}
\usepackage{graphicx,booktabs,multirow}
\usepackage{tikz}
\usetikzlibrary{patterns}
\usepackage[numbers,sort&compress]{natbib}
\usepackage{microtype}
\usepackage[hidelinks,unicode]{hyperref}
\hypersetup{
  pdftitle={[k]-Roman domination on cylindrical grids C(m) x P(n)},
  pdfauthor={Simon Brezovnik and Janez Zerovnik}
}
\newtheorem{theorem}{Theorem}
\newtheorem{proposition}[theorem]{Proposition}
\newtheorem{lemma}[theorem]{Lemma}
\newtheorem{corollary}[theorem]{Corollary}
\theoremstyle{definition}

\title[{$[k]$-Roman domination on cylindrical grids}]{$[k]$-Roman domination on cylindrical grids $C_m\Box P_n$}
\author[S. Brezovnik]{Simon Brezovnik}
\address{Faculty of Mechanical Engineering, University of Ljubljana, Ljubljana, Slovenia; Institute of Mathematics, Physics and Mechanics, Ljubljana, Slovenia}
\email{simon.brezovnik@fs.uni-lj.si}
\author[J. \v{Z}erovnik]{Janez \v{Z}erovnik}
\address{Faculty of Mechanical Engineering, University of Ljubljana, Ljubljana, Slovenia; Rudolfovo -- Science and Technology Centre, Novo Mesto, Slovenia}
\email{janez.zerovnik@fs.uni-lj.si}
\date{}
\subjclass[2020]{05C69, 05C76}
\keywords{$[k]$-Roman domination, domination, efficient domination, Cartesian product of graphs, cylindrical grids, packing number}

\begin{document}
\raggedbottom
\begin{abstract}
Roman domination and its higher-order extensions have attracted considerable attention due to their natural interpretation in terms of defensive resource allocation on networks.
The recently introduced $[k]$-Roman domination framework unifies classical Roman, double, triple, and higher-strength protection schemes by allowing each fortified vertex to provide up to $k$ levels of support.
In this paper, we investigate the $[k]$-Roman domination number $\gamma_{[k]R}(G)$ on cylindrical grids $C_m \Box P_n$.
{We relate $[k]$-Roman domination to efficient domination and establish exact values for regular graphs admitting an efficient dominating set;} as a consequence, we obtain explicit values for broad families of toroidal grids and determine exactly when the cylindrical graphs $C_m\Box P_n$ admit an efficient dominating set.
Building on these structural insights, we derive several upper bounds for $\gamma_{[k]R}(C_m \Box P_n)$ for small fixed values of $m$, accompanied by explicit labeling patterns that attain these bounds.
All obtained bounds are systematically compared, revealing parameter ranges in which different constructions dominate depending on the value of $k$ and the length of the path.
In addition, we present exact packing numbers for selected cylindrical graphs, which complement the domination results and enable further refinements via local weight reductions.
Our results extend and unify known domination-type parameters on grid-like structures and highlight new regularities that emerge as the reinforcement strength increases.
\end{abstract}
\maketitle

\section{Introduction}\label{sec:intro}

Domination in graphs is a fundamental concept in graph theory, introduced to formalize the idea of monitoring or controlling a network through a strategically selected subset of vertices.
A set \(D \subseteq V(G)\) is called a \emph{dominating set} if every vertex outside \(D\) is adjacent to at least one vertex in \(D\). 
Over the past decades, this concept has inspired numerous variants that capture different theoretical and applied scenarios—among including total domination, secure domination, and restricted domination (see, for example, the foundational works in~\cite{Haynes1998a} and \cite{Haynes1998b}).

Among the many extensions of domination parameters, \emph{Roman domination} is distinguished both by its historical motivation and its mathematical structure. 
It was introduced by Cockayne \emph{et al.}~\cite{Cockayne1985}, inspired by a military strategy attributed to the Roman Empire under Emperor Constantine the Great: each province was required to maintain its own legion or to rely on a neighboring province capable of deploying reinforcements if necessary. 
This idea translates naturally into a graph-theoretic model.

Formally, a \emph{Roman dominating function} (RDF) on a graph \(G=(V,E)\) is a labeling \(f : V(G) \to \{0,1,2\}\) such that every vertex labeled \(0\) has at least one neighbor labeled \(2\). 
The \emph{weight} of \(f\) is defined by
\[
w(f) = \sum_{v\in V(G)} f(v),
\]
and the \emph{Roman domination number} \(\gamma_R(G)\) is the minimum weight among all RDFs. 
Vertices labeled \(2\) represent heavily fortified positions, those labeled \(1\) moderately defended ones, while vertices labeled \(0\) remain undefended but supported. 
{The heavily fortified vertices are able to answer attacks to the undefended neighbors by deploying a legion thus reducing the total number of legions needed to defend all vertices.}
The parameter $\gamma_R(G)$ has since become {a} very popular topic in domination theory, and it has been studied across a broad spectrum of graph classes; for comprehensive overviews (including classical and many variant parameters) see the series of surveys on Roman domination~\cite{Chellali2021VarietiesI, Chellali2020VarietiesII, Chellali2024VarietiesIII, Chellali2025VarietiesIV, math11020351}.

Several strengthened versions of Roman domination have been introduced. 
The notion of \emph{double Roman domination}, introduced by Beeler \emph{et al.}~\cite{Beeler2016}, allows labels from the set \(\{0,1,2,3\}\), requiring that  every vertex labeled $0$ has at least
two neighbors labeled $2$ or one neighbor labeled $3$, and every vertex labeled
$1$ has at least one neighbor labeled at least $2$.
Subsequently, the \emph{triple} and \emph{quadruple} Roman domination parameters were defined as natural higher-order generalizations (see~\cite{ABDOLLAHZADEHAHANGAR2021125444,doi:10.1142/S1793830921501305}).

{In~\cite{Beeler2016}, the authors established explicit inequalities relating the double Roman domination number to the classical Roman domination number, proved the tightness of these bounds, and described structural conditions under which optimal double Roman dominating labelings require vertices of weight~3. Anu and Lakshmanan~\cite{AnuAparna2024} studied double Roman domination
in Cartesian product {of} graphs, with special emphasis on products of paths and cycles.
They investigated how the double Roman domination number behaves under the Cartesian product
operation and derived explicit bounds, and in several cases exact values, for families such as
$P_m\Box P_n$, $C_m\Box P_n$, and related cylindrical and grid-like graphs.
Their work provides a systematic treatment of Roman-type domination parameters
in product graphs and serves as an important reference point for further generalizations,
including the $[k]$-Roman domination considered in this paper.

Triple Roman domination was investigated in~\cite{ABDOLLAHZADEHAHANGAR2021125444}, where the authors introduced the parameter, established sharp general bounds, analyzed extremal configurations, and studied the structure of optimal triple Roman dominating labelings, thereby laying the groundwork for higher-strength Roman domination variants.

To unify these developments, Abdollahzadeh Ahangar \emph{et al.}~\cite{ABDOLLAHZADEHAHANGAR2021125444} proposed the general framework of \emph{\([k]\)-Roman domination}. 
A function \(f : V(G) \to \{0,1,2,\dots,k+1\}\) is a \emph{\([k]\)-Roman dominating function} ($[k]$-RDF) if for every vertex \(v\) with \(f(v)<k\),
\[
f(N[v]) \ge k + |AN(v)|,
\]
where \(AN(v)=\{u\in N(v): f(u)>0\}\) denotes the \emph{active neighborhood} of \(v\). 
The minimum weight of a \([k]\)-RDF is the \emph{\([k]\)-Roman domination number}, denoted \(\gamma_{[k]R}(G)\). 
This framework subsumes all classical variants: for \(k=1\) we obtain the classical Roman domination number, while \(k=2,3,4\) correspond to the double, triple, and quadruple Roman domination numbers, respectively.

A concept of $[k]$-Roman domination has been actively studied in recent years for various graph classes.
In~\cite{Khalili2023}, the authors derived general bounds for the $[k]$-Roman domination number and determined exact values for several specific families of graphs, including paths, cycles, complete graphs, and stars, as well as certain classes of trees, and provided characterizations of graphs attaining small $[k]$-Roman domination numbers.
In~\cite{Haghparast2022}, the focus was placed on the $[k]$-Roman domination subdivision number; the authors investigated how subdividing edges affects the $[k]$-Roman domination number, established bounds for this parameter, and determined exact values for paths, cycles, complete graphs, and stars, thereby illustrating the sensitivity of $[k]$-Roman domination to local structural modifications.
More recently, Valenzuela-Tripodoro et al.~\cite{Valenzuela2024} obtained refined upper and lower bounds for $[k]$-Roman domination, proved NP-completeness results even for restricted graph classes such as star-convex and comb-convex bipartite graphs, and derived exact values for additional families of graphs, further elucidating the structure of optimal $[k]$-Roman dominating labelings.}

We note that the literature contains three similar but mutually independent notions related to ``$k$-th extension of Roman domination''. 
Two of them, both called Roman $k$-domination but based on different definitions, were introduced respectively in~\cite{Kammerling2009ROMANKI} and~\cite{KhoshAhang2024}. 
The concept studied in this paper corresponds instead to the \([k]\)-Roman domination framework of~\cite{ABDOLLAHZADEHAHANGAR2021125444}, which unifies Roman, double, triple and higher-order Roman domination into a single coherent generalization.

{In this paper, we investigate cylindrical graphs of the form \(C_m\Box P_n\)},
and derive new general bounds together with constructive characterizations of their \([k]\)-Roman dominating functions. 
These results generalize classical Roman domination on cylindrical grids and reveal new regularities that emerge in higher-level Roman domination frameworks.

{The paper is organized as follows.
Section~\ref{sec:prelim} reviews the terminology and preliminary results
used throughout the paper.
In Section~\ref{sec:cyl}, we study $[k]$-Roman domination on cylindrical
graphs $C_m\Box P_n$. We characterize the efficient cylindrical graphs,
derive lower and upper bounds for the $[k]$-Roman domination number, and
determine the packing numbers of selected cylindrical graph families.
Finally, Section~\ref{sec:concl} concludes the paper with a summary of the
results and directions for future research.
}

\section{Preliminaries}\label{sec:prelim}

Let \(G=(V,E)\) be a finite, simple, undirected graph.
For a fixed integer \(k\ge1\), we consider labelings with values in
\(\{0,1,\ldots,k+1\}\). For a labeling
\(f:V\to\{0,1,\ldots,k+1\}\) and a set \(X\subseteq V\), we write
\[
f(X)=\sum_{x\in X}f(x).
\]

Given a vertex \(v\in V\), the \emph{active neighborhood} of \(v\) is
defined by
\[
AN_G(v)=\{u\in N(v):f(u)>0\}.
\] If the graph is clear from the context, the subscript will be omitted and $\mathrm{AN}(v)$ will be used.

A labeling \(f\) is called a \emph{\([k]\)-Roman dominating function} (\emph{\([k]\)-RDF}) if for every vertex \(v\in V\) with \(f(v)<k\),
\[
f(N[v]) \;\ge\; k + |AN(v)|,
\]
where \(N[v]=N(v)\cup\{v\}\) denotes the closed neighborhood of \(v\).
The \emph{weight} of \(f\) is \(w(f)=\sum_{v\in V} f(v)\), and the \emph{\([k]\)-Roman domination number} of \(G\) is
\[
\gamma_{[k]R}(G)=\min\{w(f) : f \text{ is a \([k]\)-RDF on }G\}.
\]
A $[k]$-Roman dominating function $f$ on $G$ is called a \emph{$\gamma_{[k]R}$-function}
if it attains the minimum weight, that is, if $w(f)=\gamma_{[k]R}(G)$.

It is often convenient to write \(f=(V_0,V_1,\dots,V_{k+1})\), where \(V_i=\{v\in V : f(v)=i\}\).
{
When discussing \([k]\)-RDFs, we often employ the tuple notation \(f=(V_0,\ldots,V_{k+1})\) and \(w(f)=\sum_{v\in V} f(v)\).
}
For \(k=1\) this framework reduces to the classical Roman domination, while general background and recent progress for the \([k]\)-case are provided in~\cite{Ahangar2021,Khalili2023, Haghparast2022}.

Two simple but useful observations will be applied repeatedly in the sequel.  
First, for every \(k\ge 2\) there exists an optimal \([k]\)-RDF with no vertices assigned label \(1\) (see~\cite{Haghparast2022,Valenzuela2024}).  
Second, many standard bounds for Roman-type parameters extend to \(\gamma_{[k]R}(G)\) with linear dependence on \(k\); whenever such bounds are needed, we cite them in the context in which they arise.

{A set $S\subseteq V(G)$ is called a \emph{packing set} of a graph $G$ if the closed
neighborhoods of distinct vertices in $S$ are pairwise disjoint, that is,
\[
N[u]\cap N[v]=\varnothing
\quad\text{for all distinct } u,v\in S.
\]
The \emph{packing number} of $G$, denoted by $\rho(G)$, is the maximum cardinality
of a packing set in $G$.}

{Mojdeh et al.~\cite{MojdehPeterinSamadiYero2020} investigated the packing number of different graph products, including the Cartesian
product. In their work, they established general upper and lower bounds for
the packing number of Cartesian, strong, and direct products in terms of the
corresponding parameters of the factor graphs, and analyzed how structural
properties of the factors influence the behavior of packing sets in the
product graph. To our knowledge, no further results specifically addressing
packing sets in Cartesian product graphs are obtained in the literature.}

{
As will be seen later, the concept of $[k]$-Roman domination is related
to the notion of efficient domination introduced by Bange, Barkauskas and
Slater~\cite{BangeBarkauskasSlater}
and surveyed in~\cite{Haynes2023}. A set $D \subseteq V(G)$ is an \emph{efficient dominating set} if for every vertex
$v \in V(G)$ there exists a unique vertex $u \in D$ such that
$v \in N[u]$. A graph is said to be \emph{efficient} if it contains an efficient dominating
set.
}

The Cartesian product \(G \Box H\) has vertex set \(V(G)\times V(H)\); vertices \((g,h)\) and \((g',h')\) are adjacent if either  
(i) \(g=g'\) and \(hh'\in E(H)\), or  
(ii) \(h=h'\) and \(gg'\in E(G)\).

We denote \(n=|V(G)|\), and write \(\delta(G)\) and \(\Delta(G)\) for the minimum and maximum degree of \(G\), respectively.  
We use the standard graph families \(P_n, C_n,\) and \(K_n\) for the path, cycle, and complete graph.  

{We write
\[
V(P_n)=\{0,1,\dots,n-1\}
\quad\text{and}\quad
V(C_m)=\{0,1,\dots,m-1\}.
\]
The Cartesian product $C_m\Box P_n$, called a \emph{cylindrical graph},
has vertex set
\[
V(C_m\Box P_n)=V(P_n)\times V(C_m).
\]
We denote a vertex by $(i,j)$, where $i$ is the path coordinate and
$j$ is the cycle coordinate. For each $i\in V(P_n)$, let
\[
F_i=\{(i,j):j\in V(C_m)\}.
\]
Thus, $F_i$ is the $i$-th layer, and it induces a copy of $C_m$.}

Background on Cartesian products and their applications to interconnection networks can be found in~\cite{Imrich}.

Finally, a graph $H$ is a \emph{spanning subgraph} of a graph $G$ if
$V(H)=V(G)$ and $E(H)\subseteq E(G)$.

\section{\texorpdfstring{\([k]\)-Roman domination on \(C_m \Box P_n\)}{[k]-Roman domination on C(m) x P(n)}}\label{sec:cyl}

In this paper, we study $[k]$-Roman domination on cylindrical grids $C_m\Box P_n$ for small fixed $m$.

{We begin by recalling a standard property of efficient dominating sets.

\begin{lemma}[{\cite[Proposition~9.3]{Haynes2023}}]\label{lema1}
Let $G$ be a graph. If $S$ and $S'$ are two efficient dominating sets of $G$,
then
\[
|S|=|S'|=\gamma(G).
\]
\end{lemma}

We next characterize the cylindrical graphs that admit an efficient
dominating set.
In contrast to the toroidal case, for which efficient domination is completely
characterized, to our knowledge no characterization is currently known for
cylindrical graphs $C_m \square P_n$. In the sequel, we determine exactly for which pairs $(m,n)$ the graph
$C_m\square P_n$ admits an efficient dominating set.

\begin{theorem}\label{thm:efficient-cylinders}
Let $m\ge 3$ and $n\ge 1$ be integers.
The cylindrical graph $C_m\square P_n$ admits an efficient dominating set if and
only if either $n=1$ and $m\equiv 0 \pmod 3$, or $n=2$ and $m\equiv 0 \pmod 4$.
\end{theorem}

\begin{proof}
Assume that $n=1$. Then $C_m\square P_1\cong C_m$. It is well known that the cycle
$C_m$ admits an efficient dominating set if and only if $m\equiv 0\pmod 3$.
Indeed, choosing every third vertex along the cycle yields a set whose closed
neighborhoods are pairwise disjoint and cover all vertices of $C_m$.

Now let $n=2$. Denote $(i,0)$ by $a_i$ and $(i,1)$ by $b_i$. If
$m\equiv0\pmod4$, then
\[
S=\{a_i:i\equiv0\pmod4\}\cup
  \{b_i:i\equiv2\pmod4\}
\]
is an efficient dominating set.

Conversely, suppose that $C_m\square P_2$ has an efficient dominating set
$S$. By rotating the cycle and, if necessary, interchanging the two layers,
we may assume that $a_0\in S$. Consider $b_1$. Among the vertices
\[
N[b_1]=\{b_0,b_1,b_2,a_1\},
\]
the vertices $b_0$, $b_1$, and $a_1$ are at distance at most two from $a_0$
and therefore cannot belong to $S$. Hence $b_2\in S$. Repeating the same
argument forces
\[
a_0,b_2,a_4,b_6,\ldots
\]
to belong to $S$, with the two layers alternating.

Let $r$ be the smallest positive integer such that $2r\equiv0\pmod m$.
After $r$ such steps we return to the cycle coordinate $0$. If $r$ were odd,
the construction would force $b_0\in S$, contradicting $a_0\in S$. Hence $r$
must be even. If $m$ is odd, then $r=m$ is odd, which, as noted above, is impossible. If $m$
is even, then $r=m/2$, and therefore $m/2$ must be even. Thus
$m\equiv0\pmod4$.

Suppose next that $n \ge 3$ and $m=3$. In  
{this} 
 case a single vertex belonging to an efficient dominating set dominates an
entire {$C_3$-layer}, 
 and efficient domination then forces the neighboring layer to contain no vertices of the dominating set.
As a consequence, two adjacent vertices in that layer need to be dominated by two vertices from the next layer, which is a contradiction.
 
Assume now that $n \ge 3$ and $m\ge 4$, and let $S$ be an efficient dominating set of $C_m\square P_n$. 
{Consider the first $C_m$-layer. 
Clearly there must be at least  one vertex of $S$ in the first layer. 
Without loss of generality, assume that $(1,0) \in S$. 
The vertex $(1,2)$ can be dominated by exactly one of the vertices $(0,2)$,  $(2,2)$, or $(1,3)$, otherwise $S$ would not be efficient.

First, assume that $(1,3)\in S$ (this possibility does not occur when $n =3$).
Consider vertices $(2,1)$ and $(2,2)$.
Neither of them can be in $S$ because $(2,1)$ has common neighbor with $(1,0)$ and $(2,2)$ has a common neighbor with $(1,3)$.
So $(2,1)$ and $(2,2)$ can only  be dominated by adjacent vertices  $(3,1)$ and $(3,2)$, which is contradiction with the assumption that $S$ is efficient.

Now assume that $(2,2)\in S$ and consider the vertices $(3,0)$ and
$(3,1)$. The vertex $(3,0)$ has a common neighbor with $(1,0)$, while
$(3,1)$ has a common neighbor with $(2,2)$, so neither of them can belong
to $S$. As in the previous case, their only possible dominators are
$(4,0)$ and $(4,1)$, respectively. Hence both vertices must belong to
$S$, which is impossible since they are adjacent.

The last case $(0,2)\in S$ is analogous to the second,  by symmetry.

We conclude that $C_m\square P_n$ for $m\ge 4$ and $n>2$ can not be efficiently dominated.
This concludes the proof.
}
\end{proof}

\subsection{\texorpdfstring{Lower bound for $\gamma_{[k]R}(C_m\Box P_n)$}{Lower bound for cylindrical grids}}  

{To assess the sharpness of the bounds,
we now establish a general lower bound for the $[k]$-Roman domination number
of regular graphs.
Since the bound involves only the order and degree of the graph,
it applies naturally to Cartesian products such as toroidal or cylindrical graphs.}

We first recall the following general lower bound obtained by
Valenzuela-Tripodoro et al.

\begin{proposition}[{\cite[Proposition~2.7]{Valenzuela2024}}]
Let $G$ be a nontrivial connected graph of order $n$ and maximum degree
$\Delta\ge k$. Then
\[
\gamma_{[k]R}(G)
\ge
\left\lceil\frac{(k+1)n}{\Delta+1}\right\rceil.
\]
\end{proposition}

For regular graphs, the restriction $\Delta\ge k$ can be omitted, as
shown in the following result.

\begin{proposition}\label{prop:LBregular}
Let $G$ be a $\Delta$-regular graph with $\Delta\ge1$.
Then
\[
\gamma_{[k]R}(G)\;\ge\; \frac{(k+1)\,|V(G)|}{\Delta+1}.
\]
\end{proposition}

\begin{proof}
Let $f$ be a $[k]$-RDF on $G$, and let
\[
B=\{v\in V(G): f(v)=k \text{ and } AN(v)=\varnothing\}.
\]
By the definition of a $[k]$-RDF, every $v\notin B$ satisfies
\[
f(N[v])\ge k+1.
\]
We claim that there exists a matching between $B$ and $N(B)$ covering
every vertex of $B$. To prove this, let $X\subseteq B$. Since $B$ is independent and $G$ is
$\Delta$-regular with $\Delta\ge1$, every subset $X\subseteq B$ has at
least $|X|$ neighbors. The vertices of $X$ are incident with
$\Delta|X|$ edges, while each vertex of $N(X)$ is incident with at most
$\Delta$ of these edges. Therefore, $|N(X)|\ge |X|$, and Hall's theorem
implies that the claimed matching exists.

Let $u \in V(G)$ be matched to some $v\in B$. Since $f(u)=0$, the vertex $u$
must have at least two active neighbors; otherwise, $v$ would be its
only active neighbor and the $[k]$-Roman condition at $u$ would fail.
Consequently,
\[
f(N[u])\ge k+2.
\]
Thus, each vertex of $B$ whose closed neighborhood has weight $k$ is
paired with a distinct vertex whose closed neighborhood has weight at
least $k+2$. Therefore,
\[
\sum_{v\in V(G)}f(N[v])\ge (k+1)|V(G)|.
\]

Since $G$ is $\Delta$-regular,
\[
\sum_{v\in V(G)}f(N[v])
=(\Delta+1)\sum_{v\in V(G)}f(v).
\]
Hence,
\[
\sum_{v\in V(G)}f(v)\ge\frac{(k+1)|V(G)|}{\Delta+1},
\]
and the result follows.
\end{proof}

Combining Proposition~\ref{prop:LBregular} with efficient domination
yields the following exact result.

\begin{proposition}\label{prop:regular-efficient}
Let $G$ be an $r$-regular graph with $r\ge1$ admitting an efficient
dominating set.
Then
\[
\gamma_{[k]R}(G)
=(k+1)\gamma(G)
=\frac{(k+1)|V(G)|}{r+1}.
\]
\end{proposition}

\begin{proof}
Let $S$ be an efficient dominating set of $G$. Since the sets
$N[s]$, $s\in S$, are pairwise disjoint and cover $V(G)$, and since
$G$ is $r$-regular, we have
\[
|V(G)|=(r+1)|S|.
\]
Moreover, Lemma~\ref{lema1} gives
\[
|S|=\gamma(G).
\]

Define $f:V(G)\to\{0,k+1\}$ by
\[
f(v)=
\begin{cases}
k+1, & \text{if } v\in S,\\
0, & \text{if } v\in V(G)\setminus S.
\end{cases}
\]
Since every vertex outside $S$ has exactly one neighbor in $S$, the
function $f$ is a $[k]$-RDF. Hence,
\[
\gamma_{[k]R}(G)\le (k+1)|S|
=\frac{(k+1)|V(G)|}{r+1}.
\]

On the other hand, Proposition~\ref{prop:LBregular}, applied with
$\Delta=r$, yields
\[
\gamma_{[k]R}(G)\ge
\frac{(k+1)|V(G)|}{r+1}.
\]
Therefore,
\[
\gamma_{[k]R}(G)
=(k+1)|S|
=(k+1)\gamma(G)
=\frac{(k+1)|V(G)|}{r+1}.
\]
\end{proof}

Applying Proposition~\ref{prop:LBregular} to toroidal grids yields the
following general lower bound.

\begin{corollary}\label{cor:toroidal-lower-bound}
For all integers $m,n\ge 3$,
\[
\gamma_{[k]R}(C_m\Box C_n) \geq \frac{(k+1)\,mn}{5}.
\]
\end{corollary}

For toroidal grids admitting an efficient dominating set, this lower
bound is attained.

\begin{corollary}
Let $m,n\ge 5$ be integers such that
$m,n\equiv 0\pmod{5}.$
Then
\[
\gamma_{[k]R}(C_m\square C_n)
=(k+1)\gamma(C_m\square C_n)
=(k+1)\frac{mn}{5}.
\]
\end{corollary}

\begin{proof}
By~\cite[Theorem~2.2]{ChelvamMutharasu2011},
$C_m\square C_n$ admits an efficient dominating set. Since
$C_m\square C_n$ is $4$-regular, Proposition~\ref{prop:regular-efficient}
applies. Moreover, the closed neighborhoods of vertices in an efficient
dominating set partition the graph into sets of size $5$. Hence,
\[
\gamma(C_m\square C_n)=\frac{mn}{5},
\]
and the result follows.
\end{proof}

In the following matrix, we illustrate a $\gamma_{[k]R}$-function for
this class of graphs. The pattern is extended periodically over the whole graph.

\begin{center}
$\displaystyle
\begin{pmatrix}
k+1 & 0   & 0   & 0   & 0 \\
0   & 0   & k+1 & 0   & 0 \\
0   & 0   & 0   & 0   & k+1 \\
0   & k+1 & 0   & 0   & 0 \\
0   & 0   & 0   & k+1 & 0
\end{pmatrix}
$
\end{center}

Theorem~\ref{thm:efficient-cylinders} similarly yields an exact value
for the efficient cylindrical family.

\begin{corollary}
Let $t\ge 1$. Then
\[
\gamma_{[k]R}(C_{4t}\square P_2)
=(k+1)\gamma(C_{4t}\square P_2)
=2t(k+1).
\]
\end{corollary}

\begin{proof}
By Theorem~\ref{thm:efficient-cylinders}, $C_{4t}\square P_2$ admits an
efficient dominating set. The graph $C_{4t}\square P_2$ is $3$-regular
and has $8t$ vertices. Thus Proposition~\ref{prop:regular-efficient} gives
\[
\gamma_{[k]R}(C_{4t}\square P_2)
=\frac{(k+1)8t}{3+1}
=2t(k+1).
\]
\end{proof}

A $\gamma_{[k]R}$-function of $C_{4t}\square P_2$ is given by the
following matrix.
\begin{center}
$\displaystyle
\begin{pmatrix}
k+1 & 0 \\
0   & 0 \\
0   & k+1 \\
0   & 0 \\
\vdots & \vdots
\end{pmatrix}
$
\end{center}

To transfer the lower bound from Corollary \ref{cor:toroidal-lower-bound} to cylindrical graphs, we compare the
$[k]$-Roman domination number under edge deletion.
The key observation is that removing edges cannot decrease the minimum
weight of a $[k]$-Roman dominating function.
This is formalized in the following lemma.

\begin{lemma}\label{lem:spanning-subgraph-monotonicity}
Let $H$ be a spanning subgraph of a graph $G$. Then
\[
\gamma_{[k]R}(G)\le \gamma_{[k]R}(H).
\]
\end{lemma}

\begin{proof}
Let $f$ be a $[k]$-Roman dominating function on $H$.
We show that $f$ is also a $[k]$-Roman dominating function on $G$.

If $f(v)\ge k$, the $[k]$-Roman domination condition in $v$ is already
satisfied in $H$ and cannot be violated by adding edges.

Next, if $f(v) \le k-1$, then every neighbor
$u\in N_G(v)\setminus N_H(v)$ with $f(u)\ge 1$ contributes at least $1$ to the
value of $f(N_G[v])$.
Hence, the total contribution of all newly added nonzero neighbors is at least
their number, and therefore
\begin{equation*}
\begin{aligned}
f\bigl(N_G[v]\bigr)
&= f\bigl(N_H[v]\bigr)
  + f\bigl(N_G[v]\setminus N_H[v]\bigr) \\[0.3em]
&\ge k + \lvert A N_H(v)\rvert
  + \sum_{x \in N_G[v]\setminus N_H[v]} f(x) \\[0.3em]
&\ge k + \lvert A N_H(v)\rvert
  + \lvert A N_G(v) \setminus AN_H(v)\rvert \\[0.3em]
&\ge k + \lvert \mathrm{AN}_G(v)\rvert .
\end{aligned}
\end{equation*}

Thus, every vertex satisfies the $[k]$-Roman domination conditions in $G$, and
$f$ is a $[k]$-Roman dominating function on $G$.
\end{proof}

For $n\ge3$, the graph $C_m\square P_n$ is a spanning subgraph of
$C_m\square C_n$. We now extend the resulting lower bound to the
cases $n=1$ and $n=2$ by considering them separately.

\begin{corollary}\label{cor:cylindrical-lower-bound}
For all integers $m\ge3$ and $n\ge1$,
\[
\gamma_{[k]R}(C_m\square P_n)
\ge \frac{(k+1)mn}{5}.
\]
\end{corollary}

\begin{proof}
First, let $n\ge3$. Since $C_m\square P_n$ is a spanning subgraph of
$C_m\square C_n$, Lemma~\ref{lem:spanning-subgraph-monotonicity} and
Corollary~\ref{cor:toroidal-lower-bound} yield
\[
\gamma_{[k]R}(C_m\square P_n)
\ge \gamma_{[k]R}(C_m\square C_n)
\ge \frac{(k+1)mn}{5}.
\]

If $n=1$, then $C_m\square P_1\cong C_m$ is $2$-regular. Hence,
Proposition~\ref{prop:LBregular} gives
\[
\gamma_{[k]R}(C_m\square P_1)
\ge \frac{(k+1)m}{3}
\ge \frac{(k+1)m}{5}.
\]

If $n=2$, then $C_m\square P_2$ is $3$-regular and has $2m$ vertices.
Thus, Proposition~\ref{prop:LBregular} gives
\[
\gamma_{[k]R}(C_m\square P_2)
\ge \frac{(k+1)2m}{4}
=\frac{(k+1)m}{2}
\ge \frac{2(k+1)m}{5}.
\]
Therefore, the result holds for every $n\ge1$.
\end{proof}}

\subsection{\texorpdfstring{Upper bounds for $\gamma_{[k]R}(C_m\Box P_n)$}{Upper bounds for cylindrical grids}}\label{subsec:upper}

\subsubsection{\texorpdfstring{Upper bounds for $\gamma_{[k]R}(C_3\Box P_n)$}{Upper bounds for C(3) x P(n)}}\label{subsec:C3Pn}

{We first analyze the case $C_3\Box P_n$, which provides a natural basis case for studying
$[k]$-Roman domination on cylindrical grids.
We begin with a construction that performs particularly well for small values of $k$
and, in the special case $k=2$, improves the currently known upper bound for
double Roman domination.}

{Before considering higher values of $k$, we recall the known exact value
for the classical Roman domination number of $C_3\Box P_n$.

\begin{proposition}[{\cite{PavlicZerovnik2012}}]
For every integer $n\ge 1$,
\[
\gamma_R(C_3\Box P_n)
=\left\lfloor\frac{3n}{2}\right\rfloor+1.
\]
\end{proposition}

We now turn to the case $k\ge2$.}

\begin{theorem}\label{prop:C3Pn}
For $k\ge 2$ and $n\ge 4$,
\begin{equation}\label{L}
\gamma_{[k]R}(C_3\Box P_n)\le nk+2 . 
\end{equation}
\end{theorem}

\begin{proof}

Define $f$ as follows. For interior layers $1\le i\le n-2$ set
\[
f(i,j)=
\begin{cases}
k,& \text{if } j\equiv i \pmod 3,\\
0,& \text{otherwise}.
\end{cases}
\]
Let us now consider the boundary layers by starting with the first layer. Define
\[
f(0,0)=k+1,\quad f(0,1)=f(0,2)=0.
\] Lastly, suppose $f(n\!-\!2,j)=k$ for $j \in \{0,1,\ldots, m-1\}$.
On the last layer $F_{n-1}$ we assign the value $k+1$ to $(n\!-\!1,j+1\!\!\! \pmod 3)$,
and $0$ to the remaining vertices of $F_{n-1}$.

Note that $f$ is a $[k]$-RDF, since the
{sum of weights of $f$  in the neighborhood of }
 each vertex labelled $0$ is at least $k + |AN(v)|$. The weight of $f$ is equal $2(k+1)+(n-2)k=nk+2$. Thus, the claim follows.
\end{proof}

In Theorem~3.3 of~\cite{AnuAparna2024}, it was shown that
\(\gamma_{[2]R}(C_3\Box P_n)\le 2n+3\).
By Theorem~\ref{prop:C3Pn}, this bound can be slightly improved, as stated in the following corollary.

\begin{corollary}
For every integer $n\ge 4$, the double {Roman} domination number of $C_3\Box P_n$ satisfies
\[
\gamma_{[2]R}(C_3\Box P_n)\le 2n+2.
\]
\end{corollary}

{In the following, we present an alternative construction that yields better bounds
for larger values of $k$.

\begin{theorem}\label{prop1-C3Cn} For $k\ge 1$ and $n\ge 4$,
\begin{equation}\label{U}
\gamma_{[k]R}(C_3 \Box P_n)
\;\le\;
3 (n-2)\left\lceil\frac{k+4}{5}\right\rceil
\;+\;
6\left\lceil
\frac{k+3-\left\lceil\frac{k+4}{5}\right\rceil}{3}
\right\rceil\le\;
\frac{3nk + 27n + 2k - 2}{5}.
\end{equation}
\end{theorem}
\begin{proof}

{For each interior layer $1 \le i \le n-2$, assign weight
\[
\left\lceil \frac{k+4}{5} \right\rceil
\]
to each vertex of the layer.
On each boundary layer (i.e., $i=0$ and $i=n-1$), assign to each of its
three vertices the weight
\[
\left\lceil
\frac{k+3-\left\lceil \frac{k+4}{5} \right\rceil}{3}
\right\rceil.
\]
}

This assignment ensures that every vertex $v$ in an interior layer satisfies
\[
f(N[v]) \ge k+4,
\]
and every boundary vertex $v$ satisfies
\[
f(N[v]) \ge k+3.
\]

Therefore, $f$ is a $[k]$-RDF.
Its total weight is at most the contribution of the interior layers,
which equals
\[
3(n-2)\left\lceil \frac{k+4}{5} \right\rceil,
\]
plus the contribution of the two boundary layers, which equals
\[
6\left\lceil
\frac{k+3-\left\lceil \frac{k+4}{5} \right\rceil}{3}
\right\rceil.
\]
Hence,
\[
w(f)
\le
3(n-2)\left\lceil \frac{k+4}{5} \right\rceil
+
6\left\lceil
\frac{k+3-\left\lceil \frac{k+4}{5} \right\rceil}{3}
\right\rceil.
\]

Using the inequalities $\lceil x \rceil \le x+1$ and
$\lceil x \rceil \ge x$, we obtain
\[
\left\lceil \frac{k+4}{5} \right\rceil \le \frac{k+4}{5}+1
\quad\text{and}\quad
\left\lceil
\frac{k+3-\left\lceil \frac{k+4}{5} \right\rceil}{3}
\right\rceil
\le
\frac{4k+11}{15}+1.
\]
Substituting these bounds yields
\[
w(f)
\le
\frac{3nk + 27n + 2k - 2}{5},
\]
which completes the proof.
\end{proof}

The construction above leaves some slack in the neighborhood sums,
as the obtained values often exceed the minimum required by the
$[k]$-Roman domination condition.
For certain values of $k$, this additive term can be reduced by increasing
the base weights to $\left\lceil\frac{k+5}{5}\right\rceil$ and then decreasing
the weight by one on vertices from a suitably chosen packing set. Therefore, the following lemma will be needed.

\begin{lemma}\label{lem:packing-C3Pn}
For $n\ge 1$, the packing number of $C_3\Box P_n$ is
\[
\rho\!\left(C_3\Box P_n\right)=\left\lceil \frac{n}{2}\right\rceil .
\]
\end{lemma}

{\begin{proof}
Since $C_3$ is a triangle, the
closed neighborhood of every vertex in $i$th layer $F_i$ contains all vertices of
$F_i$. Hence, two vertices in consecutive layers cannot belong to the
same packing set. Therefore, every packing set has cardinality at most
$\left\lceil\frac{n}{2}\right\rceil.$

For the reverse inequality, let
\[
S=\{(2q,q\pmod 3):
0\le q\le \left\lceil\frac{n}{2}\right\rceil-1\}.
\]
The selected vertices lie in every second layer. Two consecutive
selected vertices lie in layers at distance two and have different
cycle coordinates. Hence, their closed neighborhoods are disjoint.
Thus, $S$ is a packing set of cardinality
$\left\lceil\frac{n}{2}\right\rceil.$
Therefore,
\[
\rho(C_3\Box P_n)=\left\lceil\frac{n}{2}\right\rceil.
\]
\end{proof}}

This leads to the following result.

\begin{theorem}\label{prop2:C3Pn}  For $k\ge 1$ and $n\ge 4$,
\begin{equation}\label{P}
\gamma_{[k]R}(C_3\Box P_n)
\ \le\
3 (n-2) \left\lceil\frac{k+5}{5}\right\rceil
\ +\
6\left\lceil
\frac{k+4-\left\lceil \frac{k+5}{5}\right\rceil}{3}
\right\rceil
\ -\
\left\lceil \frac{n}{2}\right\rceil \le \frac{6kn+4k+55n}{10}. 
\end{equation}
\end{theorem}

\begin{proof}
 The proof follows the same lines as in the previous theorem.
For each interior layer $1\le i\le n-2$, assign weight $\left\lceil\frac{k+5}{5}\right\rceil$ to each vertex of the layer.
On each boundary layer (i.e., $i=0$ and $i=n-1$), distribute equal weight to each of its three vertices so that each boundary vertex receives weight
\[
\left\lceil
\frac{k+4-\left\lceil\frac{k+5}{5}\right\rceil}{3}
\right\rceil.
\]
The construction is chosen in such a way that every vertex $v$ satisfies
\[
f(N[v]) \ge k + |AN(v)| + 1,
\]
that is, each closed neighborhood has a surplus of one unit with respect to the
$[k]$-Roman domination condition.

Let $S$ be a maximum packing in $C_3\Box P_n$.
By Lemma \ref{lem:packing-C3Pn}, $|S|=\left\lceil \frac{n}{2}\right\rceil$.
Since the closed neighborhoods of vertices in $S$ are pairwise disjoint,
we may decrease the weight of each vertex in $S$ by one.
The surplus in the neighborhood sums guarantees that the $[k]$-Roman domination
condition remains satisfied after this local modification.

Consequently, the resulting labeling is still a $[k]$-RDF, and its total
weight is reduced by $\left\lceil \frac{n}{2}\right\rceil$.
Therefore,
\[
\gamma_{[k]R}(C_3\Box P_n)
\ \le\
3 (n-2) \left\lceil\frac{k+5}{5}\right\rceil
\ +\
6\left\lceil
\frac{k+4-\left\lceil \frac{k+5}{5}\right\rceil}{3}
\right\rceil
\ -\
\left\lceil \frac{n}{2}\right\rceil,
\]
which yields the first stated bound. Using the inequalities $\lceil x \rceil \le x+1$ and
$\lceil x \rceil \ge x$, the proof is completed.
\end{proof}

For some smaller values of $k$ the bounds obtained above can be further improved.
The following theorem summarizes this refinement.

\begin{theorem}\label{prop3:C3Pn}
Let $n\ge 4$ and $k \le 13$, $k \notin \{1,8,11,12\}$.
Set
\[
A=\left\lceil\frac{k+4}{4}\right\rceil,\qquad
B=\left\lceil\frac{-2+3A}{2}\right\rceil,\qquad
C=k+4-3A.
\]
Then
\begin{equation}\label{S}
\begin{aligned}
\gamma_{[k]R}(C_3\Box P_n)
\ \le\ &
2 (n-2)A
+4B
+2C\\
\le\ &
\frac{nk + 2k + 8n + 12}{2}.
\end{aligned}
\end{equation}
\end{theorem}

{\begin{proof}
On the interior layers of $C_3\Box P_n$, assign weights according to the
periodic pattern
\[
\left(
\begin{array}{ccccccc}
\cdots & 0 & A & A & \cdots\\
\cdots & A & 0 & A & \cdots\\
\cdots & A & A & 0 & \cdots
\end{array}
\right).
\]
On each boundary layer, let $x$ be the unique vertex whose neighbor in the
adjacent interior layer has weight $0$. Assign weight $C$ to $x$ and weight
$B$ to each of the remaining two boundary vertices.

We verify the $[k]$-Roman domination condition according to the position and
weight of a vertex. For an interior vertex not adjacent to a boundary layer,
the possible pairs
\[
\bigl(f(N[v]),|AN(v)|\bigr)
\]
are $(4A,4)$ when $f(v)=0$, and $(3A,2)$ when $f(v)=A$.

Consider now an interior layer adjacent to a boundary layer. Up to a cyclic
permutation of the coordinates, the possible pairs are
\[
(3A+C,4),\qquad (2A+B,2),\qquad (3A+B,3).
\]
For a boundary layer, the vertex of weight $C$ yields the pair
\[
(C+2B,2),
\]
whereas each vertex of weight $B$ yields the pair
\[
(A+C+2B,3).
\]

Since $C=k+4-3A$, we have $3A+C=k+4$. Moreover, for
\[
k\in\{2,3,4,5,6,7,9,10,13\},
\]
a direct verification gives
\[
3A\ge k+2,\qquad
2A+B\ge k+2,\qquad
3A+B\ge k+3,
\]
and
\[
C+2B\ge k+2,\qquad
A+C+2B\ge k+3.
\]
Also, $4A\ge k+4$ by the definition of $A$. Thus, in every case,
\[
f(N[v])\ge k+|AN(v)|,
\]
and $f$ is a $[k]\mathrm{RDF}$ of weight
\[
2(n-2)A+4B+2C.
\]

It remains to verify the second inequality. Since
\[
A=\left\lceil\frac{k+4}{4}\right\rceil,
\]
we have $4A\le k+7$. Moreover,
\[
B=\left\lceil\frac{3A-2}{2}\right\rceil\le\frac{3A}{2},
\]
and hence $8B\le12A$. Therefore,
\begin{align*}
&2\bigl(2(n-2)A+4B+2C\bigr)
-\bigl(nk+2k+8n+12\bigr)\\
&\quad=
n(4A-k-8)+8B+2k+4-20A\\
&\quad\le
4(4A-k-8)+8B+2k+4-20A\\
&\quad=
8B-4A-2k-28\\
&\quad\le
8A-2k-28\\
&\quad\le -14<0.
\end{align*}
Consequently,
\[
2(n-2)A+4B+2C
\le
\frac{nk+2k+8n+12}{2}.
\]
\end{proof}}

Figure \ref{diagramC3_1} summarizes a comparison of the upper bounds obtained from
Equations~(\ref{L})--(\ref{S}).
For each fixed $k$, we indicate the range(s) of $n\ge 4$ for which a given bound
attains the smallest value among the available estimates.
Note that the bound~\ref{S} is applicable only for $k\le 13$ with
$k\notin\{1,8,11,12\}$.
The figure presents the results explicitly for $k\le 27$; {For larger values of $k$, the smallest upper bounds among the available
estimates are given by either \eqref{U} or \eqref{P}, whose dominance
alternates according to the residue class of $k$ modulo~$5$.}

\begin{figure}
    \centering
    \includegraphics[width=0.86\linewidth]{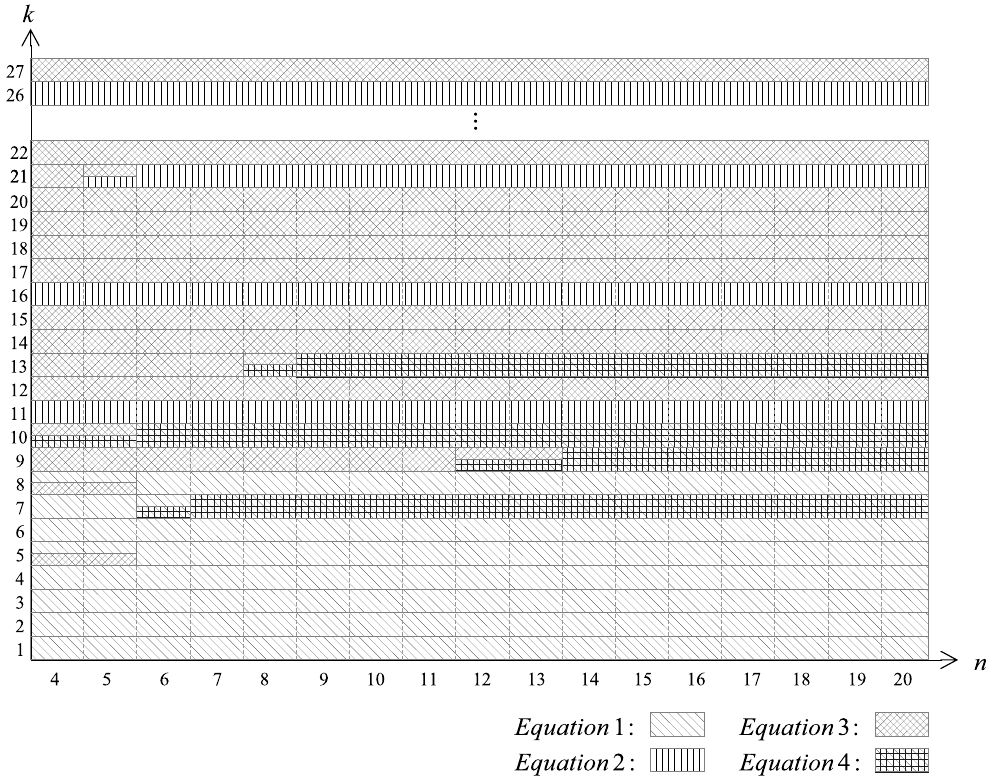}
    \caption{Best (smallest) upper bounds among \eqref{L}--\eqref{S} for $C_3\Box P_n$ for $1\le k\le 27$.}
    \label{diagramC3_1}
\end{figure}

For large $k$  the alternation between the bounds \eqref{U} and \eqref{P}
can be explained by the dependence on the ceiling terms
$\left\lceil\frac{k+5}{5}\right\rceil$ and
$\left\lceil\frac{k+4}{5}\right\rceil$.
In particular, for $k\not\equiv 1 \pmod 5$ these two values coincide, and the
packing correction $\left\lceil\frac{n}{2}\right\rceil$ makes
\eqref{P} the preferred estimate.
When $k\equiv 1 \pmod 5$, however, the term
$\left\lceil\frac{k+5}{5}\right\rceil$ is one larger than
$\left\lceil\frac{k+4}{5}\right\rceil$, and this increase dominates the packing
correction, so that \eqref{U} becomes the better bound.

\begin{table}[ht]
\centering
\renewcommand{\arraystretch}{1.2}
\begin{tabular}{c p{11.2cm}}
\hline
$k \bmod 5$ & Expected dominant bound(s) for large $n$ \\
\hline
$1$       & Equation~\eqref{U} \\
$0,2,3,4$ & Equation~\eqref{P} \\
\hline
\end{tabular}
\caption{Heuristic dominance pattern between the bounds \eqref{U} and \eqref{P}
for $C_3\Box P_n$, according to the residue class of $k$ modulo $5$.}
\label{tab:C3Pn-mod5-UP}
\end{table}

Moreover, the difference between the bounds \eqref{U} and \eqref{P} is an
explicit linear function of $n$.
Whenever the two bounds depend on different ceiling values, their leading
terms differ by a multiple of $(n-2)$, while the packing correction in
\eqref{P} contributes an additional term of order $\lceil \frac{n}{2}\rceil$.
Consequently, the gap between \eqref{U} and \eqref{P} grows proportionally
to $n$, and can therefore be made arbitrarily large by considering
sufficiently long paths.
This confirms that neither bound uniformly dominates the other and
highlights the relevance of both constructions.

\subsubsection{\texorpdfstring{Upper bounds for $\gamma_{[k]R}(C_4\Box P_n)$}{Upper bounds for C(4) x P(n)}}\label{subsec:C4Pn}

We now present our first construction for $C_4\Box P_n$, which is based on an alternating
pattern on consecutive layers combined with a suitable boundary treatment
and yields a linear upper bound in terms of $n$ for fixed $k$. Since the labeling pattern repeats with period $4$ along the $P_n$-direction,
the resulting upper bounds naturally split according to $n \bmod 4$.

\begin{theorem}\label{prop:C4Pn-simple}
Let $k\ge 1$ and $n\ge 1$. Then
\begin{equation} \label{A}
\gamma_{[k]R}(C_4\Box P_n)\le n(k+1). 
\end{equation}
\end{theorem}

\begin{proof}
Define a labeling $f:V(C_4\Box P_n)\to\{0,k+1\}$ by placing exactly one value
$k+1$ in each layer:
\[
f(i,j)=
\begin{cases}
k+1, & \text{if $i$ is even and $j=0$,}\\
k+1, & \text{if $i$ is odd and $j=2$,}\\
0,   & \text{otherwise}.
\end{cases}
\]

One can observe $f$ is a $[k]$-RDF. Indeed, every vertex $v$ with $f(v)=0$ has at
least one neighbor labeled $k+1$,
so the $[k]$-Roman domination condition holds.

Finally, since there is exactly one vertex of weight $k+1$ in each of the $n$
layers, we obtain
\[
w(f)=n(k+1),
\]
which proves the stated upper bound.
\end{proof}

As a direct consequence of the above construction, we obtain the following
upper bound for the $2$-Roman domination number.

\begin{corollary}\label{cor:C4Pn-2Roman}
For every integer $n\ge 3$,
\[
\gamma_{[2]R}(C_4\Box P_n)\le 3n.
\]
\end{corollary}

The bound from Theorem~\ref{prop:C4Pn-simple} is obtained by placing a single
vertex of weight $k+1$ in each layer, which is particularly effective for small
values of $k$ and short paths. However, for longer paths one can improve the
estimate by distributing smaller weights more evenly across the interior layers,
in the spirit of the constructions used in the previous section for $C_3\Box P_n$.
More precisely, we take a uniform base weight
$
\left\lceil\frac{k+4}{5}\right\rceil
$
on the interior vertices (following a periodic pattern), and then perform local
boundary corrections on the first and the last few layers to meet the
$[k]$-Roman domination condition. This refinement yields a better upper bound
for sufficiently large $n$, and leads to the next theorem.

\begin{theorem}\label{prop10:C4Pn} For $k\ge 1$ and $n\ge 4$,
\begin{equation}\label{C}
\gamma_{[k]R}(C_4\Box P_n)\  \le 4 (n-2) \left\lceil\frac{k+4}{5}\right\rceil  +  8\left\lceil \frac{k+3-\left\lceil\frac{k+4}{5}\right\rceil}{3}\right\rceil<
\frac{12nk + 8k + 108n - 8}{15}.   
\end{equation}
\end{theorem}

\begin{proof}
The proof is analogous to that of Theorem~\ref{prop1-C3Cn}.
We use the same assignment on the interior layers and perform
local corrections on the boundary layers to satisfy the $[k]$-Roman
domination condition. Therefore, the claimed upper bound follows.
\end{proof}

As before, the weighting scheme based on the base value
$\left\lceil\frac{k+4}{5}\right\rceil$
produces closed neighborhood sums that exceed the required threshold.
This excess can be utilized to further reduce the total weight.
Specifically, we increase the base weight to
$\left\lceil\frac{k+5}{5}\right\rceil$
and subsequently decrease the weight by one on each vertex belonging to a maximum
packing set.

To prove the next bound, we will need the following result.

\begin{lemma}\label{lem:packing-C4Pn}
For $n\ge 1$, the packing number of $C_4\Box P_n$ is
\[
\rho\!\left(C_4\Box P_n\right)=\left\lceil \frac{2n}{3}\right\rceil .
\]
\end{lemma}

{\begin{proof}
Let $S$ be a packing set in $C_4\Box P_n$. Since any two vertices in
the same layer have intersecting closed neighborhoods, every layer
contains at most one vertex of $S$.

Suppose that $S$ contains vertices in three consecutive layers
$F_i,F_{i+1},F_{i+2}$. If $(i,a)$ and $(i+1,b)$ belong to $S$, then
$a$ and $b$ must be opposite vertices of $C_4$; otherwise their closed
neighborhoods intersect. Similarly, if $(i+1,b)$ and $(i+2,c)$ belong
to $S$, then $b$ and $c$ must be opposite. Consequently, $a=c$.
However, the closed neighborhoods of $(i,a)$ and $(i+2,a)$ intersect
at $(i+1,a)$, a contradiction. Thus, among any three consecutive
layers, at most two contain vertices of a packing set.

Writing $n=3q+r$, where $r\in\{0,1,2\}$, this gives
\[
|S|\le 2q+r=\left\lceil\frac{2n}{3}\right\rceil.
\]

For the reverse inequality, define
\[
T=
\{(3q,0):0\le 3q\le n-1\}
\cup
\{(3q+1,2):0\le 3q+1\le n-1\}.
\]
Within each pair of consecutive selected layers, the selected cycle
coordinates are opposite. Between two consecutive pairs, the selected
layers are at distance two and their cycle coordinates are different.
Therefore, the closed neighborhoods of distinct vertices in $T$ are
pairwise disjoint, so $T$ is a packing set. Moreover,
$|T|=\left\lceil\frac{2n}{3}\right\rceil.$
Hence,
\[
\rho(C_4\Box P_n)=\left\lceil\frac{2n}{3}\right\rceil.
\]
\end{proof}

The packing set $T$ is illustrated in Figure~\ref{fig:pattern-C4Pn}.

\begin{figure}[ht!]
\centering
\begin{tikzpicture}[scale=0.75]

\foreach \i in {1,...,6}{
  \foreach \r in {0,1,2,3}{
    \draw (\i,-\r) rectangle (\i+1,-\r-1);
  }
}

\foreach \i in {0,...,5}{
  \node at (\i+1.5,0.6) {\small $F_{\i}$};
}

\foreach \r in {0,1,2,3}{
  \node[anchor=east] at (0.9,-\r-0.5) {\small $\r$};
}

\node at (1.5,-0.5) {\Large $\ast$}; 
\node at (2.5,-2.5) {\Large $\ast$}; 

\node at (4.5,-0.5) {\Large $\ast$}; 
\node at (5.5,-2.5) {\Large $\ast$}; 

\node at (8.2,-2) {\Large $\cdots$};

\end{tikzpicture}
\caption{A periodic packing pattern in $C_4\Box P_n$ attaining
$\lceil \frac{2n}{3}\rceil$. The stars indicate the vertices selected in the packing set in each layer.
}
\label{fig:pattern-C4Pn}
\end{figure}}

Since the closed neighborhoods of vertices in a packing set are pairwise disjoint,
this adjustment does not violate the $[k]$-Roman domination condition.
Using the fact that $\rho(C_4\Box P_n)=\left\lceil\frac{2n}{3}\right\rceil$,
we obtain the following result.

\begin{theorem}\label{prop11:C4Pn}
For $k\ge 1$ and $n\ge 4$,
\begin{equation}\label{B}
{
\gamma_{[k]R}(C_4\Box P_n)
}
\ \le\
4 (n-2) \left\lceil\frac{k+5}{5}\right\rceil
\ +\
8\left\lceil \frac{k+4-\left\lceil\frac{k+5}{5}\right\rceil}{3}\right\rceil
\ -\
\left\lceil \frac{2n}{3}\right\rceil
\ \le\
\frac{12nk+8k+110n}{15}.
\end{equation}
\end{theorem}

\begin{proof}
The proof follows the same construction and arguments as in the proof of
Theorem~\ref{prop2:C3Pn} with Equation~\ref{P}.
Therefore, it is omitted.
\end{proof}

As in the preceding section, a refined estimate can be derived for restricted ranges of $k$. The following theorem provides such an improvement for $k \le 13$, excluding a small set of exceptional values.

\begin{theorem}\label{prop3:C4Pn} Let $n\ge 4$ and $k \le 13$, $k \notin \{1,8,11,12\}$.
Set
\[
A=\left\lceil\frac{k+4}{4}\right\rceil,\qquad
B=\left\lceil\frac{-2+3A}{2}\right\rceil,\qquad
C=k+4-3A.
\]
Then
\begin{equation}\label{D}
\begin{aligned}
{
\gamma_{[k]R}(C_4\Box P_n)
} 
\ \le\ &
3 (n-2)A
+6B
+2C \\
\le\ &
\frac{3nk+5k+24n+32}{4}.
\end{aligned}
\end{equation}
\end{theorem}

\begin{proof}
The proof is analogous to that of Theorem~\ref{prop3:C3Pn}.
The same periodic assignment on the interior layers is used, with the boundary
adjustment modified accordingly, which yields the weight
$3(n-2)A+6B+2C$.
The final inequality follows by applying $\lceil x\rceil\le x+1$ to eliminate the
ceiling functions.
For brevity, we omit the details.
\end{proof}

Figure~\ref{diagramC4_1} presents a comparison of the upper bounds
given in Theorems~\ref{prop:C4Pn-simple}--\ref{prop3:C4Pn}
for the graph $C_4\Box P_n$, for $1\le k\le 20$ and all $n\ge 4$.
For small values of $k$, the simple linear bound~\eqref{A} provides the best estimate.
As $k$ increases, more refined constructions yield improved upper bounds.
In particular, for $k=15$ and $n \leq 10$ the bound \eqref{A} is the lowest. For $k=15$ and $n=11,12$ the bounds~\eqref{A} and \eqref{B} give equal values, and for $k=15$ and $n \geq 13$ the bound~\eqref{B} outperforms other two bounds. For $k=16$ the bound~\eqref{C} eventually outperforms the linear bound~\eqref{A}.
However, there is no single bound that dominates for all sufficiently large values
of $k$; instead, the bounds~\eqref{C} and~\eqref{B} alternate in dominance as $k$
varies.
Although the bound~\eqref{D} is valid for certain values of $k$, it does not attain
the smallest value for any possible $k$.

\begin{figure}
    \centering
    \includegraphics[width=0.86\linewidth]{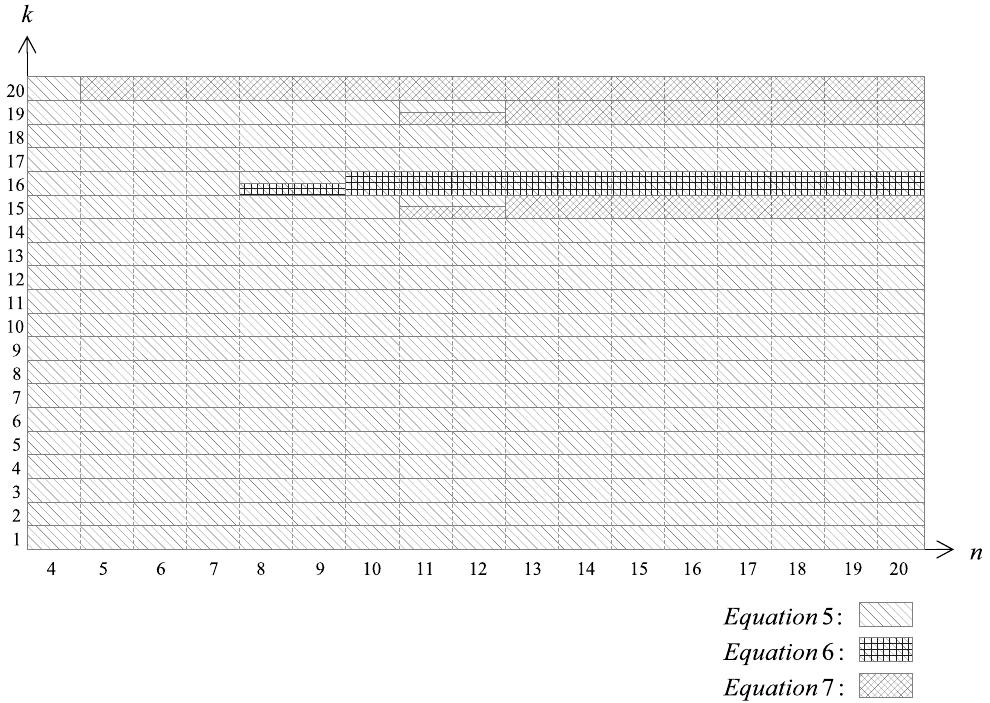}
    \caption{Best (smallest) upper bounds among \eqref{A}--\eqref{D} for $C_4\Box P_n$
for $1\le k\le 20$.}
    \label{diagramC4_1}
\end{figure}

{Up to $k=52$, the bound \eqref{A} remains the smallest among the considered
upper bounds for at least one of the small values $n=4$ or $n=5$; from
$k\ge 53$ onward, the smallest upper bounds among those considered are given
exclusively by \eqref{C} and \eqref{B}, see Figure~\ref{diagramC4_2}.} For large $n$ the alternation between the bounds \eqref{C} and \eqref{B} is explained by the
dependence on the ceiling terms $\left\lceil\frac{k+5}{5}\right\rceil$ and
$\left\lceil\frac{k+4}{5}\right\rceil$.
In particular, for $k\not\equiv 1 \pmod 5$ these two values coincide, and the
packing correction $\left\lceil\frac{2n}{3}\right\rceil$ makes \eqref{B} the
preferred estimate on sufficiently long paths.
When $k\equiv 1 \pmod 5$, however, the term $\left\lceil\frac{k+5}{5}\right\rceil$
is one larger than $\left\lceil\frac{k+4}{5}\right\rceil$, and this increase may
outweigh the packing correction, so that \eqref{C} becomes the better bound. see Table \ref{tab:C4Pn-mod5}.

\begin{figure}
    \centering
    \includegraphics[width=\linewidth]{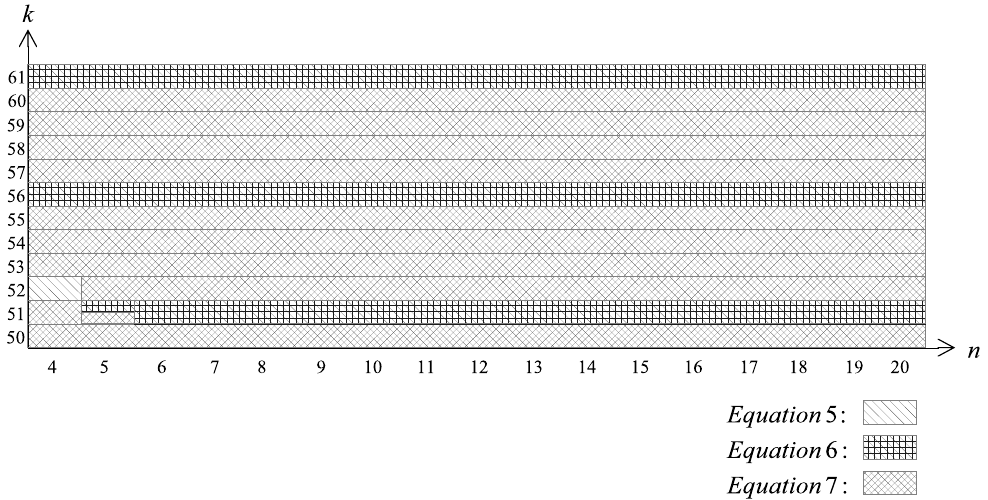}
    \caption{Best (smallest) upper bounds among \eqref{A}--\eqref{D} for $C_4\Box P_n$
for $50\le k\le 61$.}
    \label{diagramC4_2}
\end{figure}

\begin{table}[ht]
\centering
\renewcommand{\arraystretch}{1.2}
\begin{tabular}{c p{11.2cm}}
\hline
$k \bmod 5$ & Expected dominant bound(s) \\
\hline
$1$       & Equation~\eqref{C} \\
$0,2,3,4$ & Equation~\eqref{B} \\

\hline
\end{tabular}
\caption{Heuristic dominance pattern of the bounds \eqref{C} and \eqref{B}
according to the residue class of $k$ modulo $5$.}
\label{tab:C4Pn-mod5}
\end{table}

As before, the difference between both bounds is a
linear function of $n$. Hence, the gap between \eqref{C} and \eqref{B} can be made arbitrarily large by considering sufficiently long paths.

}

\section{Conclusion}\label{sec:concl}

{In this paper, we investigated the $[k]$-Roman domination number of cylindrical
graphs of the form $C_m\Box P_n$, with a primary focus on the cases $m=3$ and $m=4$.
For these graphs, we developed several explicit labeling constructions that yield
new and improved upper bounds, and we identified the ranges of parameters for which
{each construction yields the smallest upper bound among those considered.}
In the case $C_3\Box P_n$, we provided a detailed comparison of multiple bounds,
highlighting the influence of both the parameter $k$ and the length of the path
on the quality of the resulting estimates.
For $C_4\Box P_n$, we obtained simple linear bounds and showed how they can be refined
using techniques analogous to those employed in the case $m=3$.

In addition to upper bounds, we established a general lower bound for the
$[k]$-Roman domination number of regular graphs.
This bound depends only on the order and degree of the graph and therefore applies
naturally to Cartesian products such as cylindrical and toroidal graphs.
In future work, this lower bound could be combined with the constructions presented
in this paper to yield exact values, for example in the investigation of the case $m=5$.
{Based on the observed patterns, it seems plausible that a sharp upper bound
can be determined for $C_5\Box P_n$ and subsequently extended to graphs of
the form $C_{5\ell}\Box P_n$ via periodic constructions.}
Such results could serve as a foundation for a more systematic understanding of
$[k]$-Roman domination on $C_m\Box P_n$ for general values of $m$.

Another important direction concerns the development of improved lower bounds.
Although the general lower bound for regular graphs provides a useful baseline,
there remains a gap between upper and lower bounds in several cases.
Bridging this gap, either by refining existing arguments or by introducing new
techniques tailored to specific graph families, represents an intriguing open
problem.
Finally, many of the constructions considered here are expected to extend, with
only minor local adjustments, to toroidal graphs of the form $C_m\Box C_n$.
A systematic study of $[k]$-Roman domination on such graphs, as well as on general families of graph bundles, may lead to further exact results and a deeper understanding of domination-type parameters.
}

\section*{Funding}
The first author (S.B.) acknowledges financial support from the
Slovenian Research and Innovation Agency (ARIS) through Research
Programme P1-0297 and Research Project J1-70016.
The second author (J.Z.) was partially supported by ARIS through the annual work program of Rudolfovo
and by the research grants P2-0248, L1-60136.
Also supported in part by Horizon Europe project Quantum Excellence Centre for Quantum-Enhanced Applications, QEC4QEA.

\section*{Declaration of competing interest}
The authors declare that they have no known competing financial interests or personal relationships
that could have appeared to influence the work reported in this paper.

\bibliographystyle{unsrtnat}
\bibliography{references}

\end{document}